\documentclass[a4paper]{llncs}
\usepackage{textcomp}
\usepackage{amssymb}
\usepackage{verbatim}
\usepackage{array}
\usepackage{latexsym}
\usepackage{enumerate}
\usepackage{amsmath}
\usepackage{amsfonts}
\usepackage{color}
\usepackage[english]{babel}
\usepackage[notref,notcite]{showkeys}   
\usepackage{comment}
\usepackage{units}

\newtheorem{result}[theorem]{Result}
\newtheorem{prop}[theorem]{Proposition}

\newtheorem{rem}[theorem]{Remark}

\def\cF{\mathcal F}

\def\cH{\mathcal H}

\def\cX{\mathcal X}
\def\cY{\mathcal Y}

\def\PG{{\rm{PG}}}

\def\fq{{\mathbb F}_q}
\def\Fs{{\mathbb F}_{q^2}}

\newcommand{\PGU}{\mbox{\rm PGU}}

\newcommand{\aut}{\mbox{\rm Aut}}

\newcommand{\go}{\omega}

\newcommand{\ha}{{\textstyle\frac{1}{2}}}

\title{Galois subcovers  of the Hermitian curve in characteristic $p$ with respect to subgroups of order $dp$ with $d\not=p$ prime}

\author{Arianna Dionigi\inst{1} \and Barbara Gatti\inst{2}}
\institute{Università di Firenze\
\email{arianna.dionigi@unifi.it}
\and
Università del Salento\
\email{barbara.gatti@unisalento.it}}

\begin{document}

\maketitle

\begin{abstract}
A problem of current interest, also motivated by applications to Coding theory, is to find explicit equations for \textit{maximal} curves, that are
projective, geometrically irreducible, non-singular curves defined over a finite field $\mathbb{F}_{q^2}$ whose number of $\mathbb{F}_{q^2}$-rational points attains the Hasse-Weil upper bound of $q^2+2\mathfrak{g}q+1$ where $\mathfrak{g}$ is the genus
of the curve $\mathcal{X}$. For curves which are Galois covered of the Hermitian curve, this has been done so far ad hoc, in particular  in the cases where the Galois group has prime order and also when has order the square of the characteristic. In this paper we obtain explicit equations of all Galois covers of the Hermitian curve with Galois group of order $dp$ where $p$ is the characteristic of $\mathbb{F}_{q^2}$ and $d$ is prime other than $p$. We also compute the generators of the Weierstrass semigroup at a special $\mathbb{F}_{q^2}$-rational point of some of the curves, and discuss some possible positive  impacts on the minimum distance problems of AG-codes.
\end{abstract}
\vspace{0.5cm}\noindent {\em Keywords}:
maximal curves, function fields, Galois cover, Weierstrass semigroup, AG-code
\vspace{0.2cm}\noindent

\vspace{0.5cm}\noindent {\em Subject classifications}:
\vspace{0.2cm}\noindent  14H37, 14H05.

\section{Introduction}

Curves with many points over a finite field have intensively been investigated also by their connections to Coding theory, Cryptography, Finite geometry, and shift register sequences. In this context, the most important family consists of the maximal curves, that is, curves defined over the finite field $\mathbb{F}_{q^2}$, where $q=p^h$ and $p$ is its charactreristic, which attain the famous Hasse-Weil upper bound. 
The Hermitian curve is the best known maximal curve and it is also the most useful for applications, especially in the study of algebraic geometry codes, shortly AG-codes.
Actually, many other maximal curves derive from the Hermitian curve since any $\mathbb{F}_{q^2}$-subcover of a maximal curve is still maximal over the same field. If such a  $\mathbb{F}_{q^2}$-subcover is a Galois subcover with Galois group $G$ then the arising curve is named the quotient curve of the Hermitian curve with respect to $G$.
Up to group isomorphism, the $\mathbb{F}_{q^2}$-automorphism group of the Hermitian curve is the $3$-dimensional projective unitary group $\PGU(3,q)$ which has plenty of subgroups; see \cite{hoffer1972}. This motivated the systematic study of the quotients curves of the Hermitian curve  which was eventually initiated in the seminal paper of Garcìa, Stichtenoth and Xing \cite{GSX}.  Ever since important progress has been made in the study of the spectrum of the possible genera of the quotients of the Hermitian curve over a given finite field; see \cite[Chapter 10]{HKT}. Nevertheless, the problem of determining explicit equations for such curves, which is a relevant issue for applications, remains largely open. In fact, this problem has so far been solved by ad hoc methods, apart form the cases where the Galois group has either prime order; see \cite{cossidente-korchmaros-torres2000}, or its  order equals the square of the characteristic; see \cite{GK}.\\
In this paper we determine explicit equations for each quotient curve of the Hermitian curve whose Galois group has order $dp$ where $p$ is the characteristic of  $\mathbb{F}_{q^2}$ and $d$ is prime other than $p$. We also compute the Weierstrass semigroup at some $\mathbb{F}_{q^2}$-rational point of those curves, and discuss possible positive impacts on the minimum distance problems of AG-codes.
\begin{theorem}
\label{th1} In the $\mathbb{F}_{q^2}$-automorphism group $G\cong \PGU(3,q)$  of the Hermitian curve $\cH_q$ defined over $\mathbb{F}_{q^2}$ with $q=p^h$  and $p\ge 5$, let $H$ be a subgroup of order $dp$ where $d\ge 5$ is a prime number other than $p$. Let $\bar{\cH}_q=\cH_q/H$ be the quotient curve of $\cH_q$ with respect to the subgroup $H$. Then, up to an $\mathbb{F}_{q^2}$-isomorphism, one of the following cases occurs. 
\begin{itemize}
\item[(I)] If $H=C_p\times C_d$ then $\bar{\cH}_q$ has genus  $$\mathfrak{g}=\frac{1}{2d}(q-d+1)\left (\frac{q}{p}-1\right )$$ and equation
\begin{equation}
\label{eqthI}
\sum_{i=0}^{h-1}Y^{p^{i}} +\go X^{\nicefrac{(q+1)}{d}}=0\quad with\quad \omega^{q-1}=-1\quad and\quad d\mid (q+1).
\end{equation}
\item[(II)] If $H=C_p\rtimes C_d$ and $C_p$ is in the center in a Sylow $p$-subgroup of $G$, then $\bar{\cH}_q$ has genus $$\mathfrak{g}=\frac{1}{2} \frac{q}{d}\left (\frac{q}{p}-1\right )$$ and equation 
\begin{equation}
\label{eqthII}
\go X^{\nicefrac{(q-1)}{d}}-A(X,Y)=0 \quad with\quad \omega^{q-1}=-1\quad and \quad d\mid (p-1).
\end{equation}
where 
$$A(X,Y)=Y+X^{\nicefrac{2(p-1)}{d}}Y^p+\cdots+X^{\nicefrac{2(p^{h-1}-1)}{d}}Y^{\nicefrac{q}{p}}.$$
\item[(III)] If $H=C_p\rtimes C_d$ but $C_p$ is not in the center in a Sylow $p$-subgroup of $G$, then $\bar{\cH}_q$ has genus $$\mathfrak{g}=\frac{q}{2dp}(q-1)$$ and equation 
\begin{equation}
\label{eqthIII}
\left (\frac{Y^2}{X^d}\right )^{\nicefrac{(q-1)}{d}}+1-A(X,Y)=0\quad d\mid (p-1)
\end{equation}
where 
\begin{equation*}
    A(X,Y)=\sum_{i=0}^{h-1}\sum_{j=0}^{h-1}\left(\frac{Y^2}{X^d}\right)^{\nicefrac{(p^i-1)}{2d}}\left(\frac{Y^2}{X^d}\right)^{\nicefrac{(p^j-1)}{2d}}X^{\nicefrac{(p^i+p^j)}{2}}.
\end{equation*}
\end{itemize}
\end{theorem}

Our Notation and terminology are standard; see \cite{HKT,stich1993,huppertI1967}. In particular, $q$ always stands for a power of $p$, namely $q=p^h$. We mostly use the language of function field theory rather than that of algebraic geometry.

\section{Background}

Let $\cX$ be a projective, non-singular, geometrically irreducible, algebraic curve of genus $\mathfrak{g}\geq 2$ embedded in an $r$-dimensional projective space $\PG(r,\mathbb{F}_\ell)$ over a finite field of order $\ell$ of characteristic $p$.  Let $\mathbb{F}_\ell(\cX)$ be the function field of $\cX$ which is an algebraic function field of transcendency degree one with constant field $\mathbb{F}_\ell$. As it is customary, $\cX$ is viewed as a curve defined over the algebraic closure $\mathbb{F}$ of $\mathbb{F}_\ell$. Then the function field $\mathbb{F}(\cX)$ is the constant field extension of $\mathbb{F}_\ell(\cX)$ with respect to field extension $\mathbb{F}|\mathbb{F}_\ell$. The automorphism group $\aut(\cX)$ of $\cX$ is defined to be the automorphism group of $\mathbb{F}(\cX)$ fixing every element of $\mathbb{F}$. It has a faithful permutation representation on the set of all points $\cX$ (equivalently on the set of all places of $\mathbb{F}(\cX))$. The automorphism group $\aut_{\mathbb{F}_\ell}(\cX)$ of $\mathbb{F}_\ell(\cX)$ is a subgroup of $\aut(\cX)$. In particular, the action of $\aut_{\mathbb{F}_\ell}(\cX)$ on the $\mathbb{F}_\ell$-rational points of $\cX$ is the same as on the set of degree $1$ places of $\mathbb{F}_\ell(\cX)$.
Let $G$ be a finite subgroup of $\aut_{\mathbb{F}_{\ell}}(\cX)$. The \emph{Galois subcover} of $\mathbb{F}_\ell(\cX)$ with respect to $G$ is the fixed field of $G$, that is, the subfield ${\mathbb{F}_{\ell}}(\cX)^G$ consisting of all elements of $\mathbb{F}_{\ell}(\cX)$ fixed by every element in $G$. Let $\cY$ be a non-singular model of ${\mathbb{F}_{\ell}}(\cX)^G$, that is,
a projective, non-singular, geometrically irreducible, algebraic curve with function field ${\mathbb{F}_{\ell}}(\cX)^G$. Then $\cY$ is the \emph{quotient curve of $\cX$ by $G$} and is denoted by $\cX/G$. The covering $\cX\mapsto \cY$ has degree equal to $\mid G\mid$ and the field extension $\mathbb{F}_{\ell}(\cX)|\mathbb{F}_{\ell}(\cX)^G$ is Galois.
If $P$ is a point of $\cX$, the stabilizer $G_P$ of $P$ in $G$ is the subgroup of $G$ consisting of all elements
fixing $P$.
\begin{result}\cite[Theorem 11.49(b)]{HKT}
\label{resth11.49b} All $p$-elements of $G_P$ together with the identity form a normal subgroup $S_P$ of $G_P$ so that $G_P=S_P\rtimes C$, the semidirect product of $S_P$ with a cyclic complement $C$.
\end{result}
\begin{result}\cite[Theorem 11.129]{HKT}
\label{resth11.129} If $\cX$ has zero Hasse-Witt invariant then every non-trivial element of order $p$ has a unique fixed point, and hence no non-trivial element in $S_P$ fixes a point other than $P$.
\end{result}
A useful corollary of Result \ref{resth11.129} is the following.
\begin{result}
\label{lem15042023} Let $\cX$ be an  $\mathbb{F}_\ell$-rational curve whose number of $\mathbb{F}_\ell$-rational points is $N\ge 2$. If $\cX$ has zero Hasse-Witt invariant and $S$ is a $p$-subgroup of $\aut_{\mathbb{F}_\ell}(\cX)$ then $S$ fixes a unique point and $|S|$ divides $N-1$.
\end{result}

The following result is due to Stichtenoth \cite{stichtenoth1973I}.
\begin{result} \cite[Theorem 11.78(i)]{HKT}
\label{sti1} Let $H$ be a $p$-subgroup of $\mathbb{F}(\cX)$ fixing a point. If $|S|$ is larger than the genus of $\mathbb{F}(\cX)$ then the Galois subcover of $\mathbb{F}(\cX)$ with respect to $H$ is rational.
\end{result}
 From now on let $\ell=q^2$ with $q=p^h$ and assume that $\cX$ is a $\mathbb{F}_{q^2}$-maximal curve.

  The following result follows from \cite[Lemma1]{RS}.
\begin{result}
\label{zeroprank} All $\mathbb{F}_{q^2}$-maximal curves have zero Hasse-Witt invariant.
\end{result}

The following result is commonly attributed to Serre.
\begin{result}\cite[Theorem 10.2]{HKT}
\label{resth10.2} For every subgroup $G$ of $\aut_{\mathbb{F}_{q^2}}(\cX)$, the quotient curve $\cX/G$ is also  $\mathbb{F}_{q^2}$-maximal.
\end{result}
We also use the classification of all groups whose order is the product of two distinct primes.
\begin{result}
\label{res07122023} Suppose $u$ and $v$ are distinct prime numbers with $u<v$. Then, there are two possibilities for groups $G$ of order $uv$:
\begin{itemize}
\item[(I)] If $u\nmid (v-1)$  then $G$ is a cyclic group.
\item[(II)] If $u\mid (v-1)$, then either $G$ is a cyclic group, or $G$ is a semidirect product $C_v \rtimes C_u$.
\end{itemize}
\end{result}

\subsection{The function field of the Hermitian curve}
In this subsection we collect some useful results about the function field of the Hermitian curve and its Galois subcovers.
The affine equation $Y^q+Y=X^{q+1}$ of an $\mathbb{F}_{q^2}$-rational curve is the  usual canonical form of the Hermitian curve $\cH_q$ with function field is $\mathbb{F}_{q^2}(x,y)$ where $y^q+y-x^{q+1}=0$. The equation  $Y^q-Y+\omega X^{q+1}=0$ with $\omega\in \mathbb{F}_{q^2}$ such that $\omega^{q-1}=-1$ is another useful equation of $\cH_q$.
We exploit numerous known results on the $\mathbb{F}$-automorphism group $\aut(\mathbb{F}(\cH_q))$ of $\cH_q$. For more details, the Reader is referred to  \cite{hoffer1972,huppertI1967}.
\begin{result}\cite[Theorem 12.24(iv), Proposition 11.30]{HKT}
\label{sect12.3} $\aut(\mathbb{F}(\cH_q))=\aut(\mathbb{F}_{q^2}(\cH_q))\cong \PGU(3,q)$. Moreover, $\aut(\mathbb{F}(\cH_q))$ acts on the set of all $\mathbb{F}_{q^2}$-rational points of $\cH_q$ as $\PGU(3,q)$ in its natural doubly transitive permutation  representation of degree $q^3+1$ on the isotropic points of the unitary polarity of the projective plane $PG(2,\mathbb{F}_{q^2})$. 
\end{result}
The maximal subgroups of $\PGU(3,q)$ were determined by Mitchell in 1911, see Hoffer \cite{hoffer1972}. Let $S_p$ be a Sylow $p$-subgroup of $\aut(\mathbb{F}_{q^2}(\cH_q))=\PGU(3,q)$. From Results \ref{sect12.3}, it may be assumed up to conjugacy that the unique fixed point of $S_p$ is the point at infinity $Y_\infty$ of $\cH_q$. The following result describes the structure of the stabiliser of $Y_\infty$ in $\aut(\mathbb{F}_{q^2}(\cH_q))$.
\begin{result}
\label{struct} Let the Hermitian function field be given by its canonical form $\mathbb{F}_{q^2}(x,y)$ with $y^q+y-x^{q+1}=0$. Then the  stabiliser $G$ of $Y_\infty$ in $\aut(\mathbb{F}(\cH_q))$ consists of all maps
\begin{equation}
 \label{eq250523}
 \psi_{a,b,\lambda}:\,(x,y)\mapsto (\lambda x+a, a^q \lambda x +\lambda^{q+1}y +b)
 \end{equation}
 where
 \begin{equation}
 \label{eqA250523}
  a \in \mathbb{F}_{q^2},\,\,
 \lambda\in \mathbb{F}_{q^2}^*,\,\,b^q+b=a^{q+1}.
 \end{equation}
 In particular, $G=S_p\rtimes C$ where $S_p=\{\psi_{a,b,1}| b^q+b=a^{q+1},a,b\in \mathbb{F}_{q^2}\}$ and $C=\{\psi_{0,0,\lambda}|\lambda\in \mathbb{F}_{q^2}^*\}$.
 \end{result}
 A direct computation by induction on $i$ shows that for $1\le i \le p$
 \begin{equation}
 \label{eq08122023} \psi_{a,b,1}^i=\psi_{ia,a^{q+1}(i^2-i)/2+ib,1}.
 \end{equation}
 For more about Result \ref{struct} see \cite{GSX}, Section 4.
 \begin{rem}
\label{rem250523} {\em{Changes of the generators $x,y$ of the Hermitian function field  $\mathbb{F}_{q^2}(x,y), y^q+y-x^{q+1}=0$ provide another canonical form. For our purpose, a useful change is $\tau:(x,y)\rightarrow (\omega x, -\omega y)$ where $\omega^{q-1}=-1$, and the arising canonical form is $y^q-y+\omega x^{q+1}=0$. }} 
Then the elements in the stabilizer $G$ of $Y_\infty$ in $\aut(\mathbb{F}(\cH_q))$ are of the form:
\begin{equation}
\label{eaQ08122023}
\varphi_{a,b,\lambda}:\,(x,y)\mapsto (\lambda x+a, a^q \lambda \omega x +\lambda^{q+1}y +b)
\end{equation}
where (\ref{eqA250523}) is replaced by 
$$ a\in \mathbb{F}_{q^2},\,\,\lambda\in \mathbb{F}_{q^2}^*,\,\,b^q-b=-\omega a^{q+1}.$$
A direct computation by induction on $i$ shows that for $1\le i \le p$
 \begin{equation}
 \label{eq081220232} \varphi_{a,b,1}^i=\varphi_{ia,a^{q+1}\omega(i^2-i)/2+ib,1}.
 \end{equation}
For more about Result \ref{rem250523} see \cite{GSX}, Section 4.
 \end{rem}
 
\noindent From Result \ref{struct}, $S_p$ is the (unique) Sylow $p$-subgroup of the stabiliser of $Y_\infty$ in $\aut(\mathbb{F}_{q^2}(\cH_q))$.
\begin{result}
\label{conjcl} $S_p$ has the following properties.
\begin{itemize}
\item[(I)] The center $Z(S_p)$  of $S_p$ has order $q$ and it consists of all maps $\psi_{0,b,1}$ with $b^q+b=0, b\in\mathbb{F}_{q^2}$. Also, $Z(S_p)$ is an elementary abelian group of order $p$.
\item[(II)]  The non-trivial elements of $S_p$ form two conjugacy classes in  the stabiliser of $Y_\infty$ in $\aut(\mathbb{F}_{q^2}(\cH_q))$, one comprises all non-trivial elements of $Z(S_p)$, the other does the remaining $q^3-q$ elements.
\item[(III)] The elements of $G$ other than those in $Z(S_p)$ have order $p$, or $p^2=4$ according as $p>2$ or $p=2$.
\end{itemize}
\end{result}

For completeness, we provide a proof for the classification of subgroups of $\PGU(3,q)$ of order $dp$. We use the canonical form $y^q-y+\omega x^{q+1}=0$ with $\omega^{q-1}=-1$.
The Galois subcovers of $\mathbb{F}_{q^2}(\cH_q)$ with respect to a subgroup $H$ of prime order or when its  order equals the square of the characteristic were thoroughly classified in \cite{cossidente-korchmaros-torres2000} and \cite{GK} respectively. For the case $|H|=dp$, the classification is reported in the following result.
\begin{theorem}\label{dp} 
Let $p$ and $d$ two distinct prime numbers both larger than $3$. Then, up to conjugacy in $\PGU(3,q)$, there exist at most three subgroups of order $dp$ in $\PGU(3,q)$, one is cyclic and the other two are semidirect products of $C_p\rtimes C_d$ with $p<d$.
They are subgroups of the stabiliser of $Y_\infty$ in $\aut(\mathbb{F}_{q^2}(\cH_q))$ where

\begin{enumerate}
\item[(I)] $G=\Sigma_p\times\Sigma_d$ with $\Sigma_p=\langle \varphi_{0,1,1}\rangle$ and $\Sigma_d=\langle \varphi_{0,0,\lambda}\rangle$ with $\lambda^d=1$, $d|(q+1)$;
\item [(II)] $G=\Sigma_p\rtimes\Sigma_d$ with $\Sigma_p=\langle \varphi_{0,1,1}\rangle$ and $\Sigma_d=\langle \varphi_{0,0,\lambda}\rangle$ with $\lambda^d=1$, $d|(p-1)$;
\item[(III)] $G=\Sigma_p\rtimes\Sigma_d$ with $\Sigma_p=\langle \varphi_{1,\omega/2,1}\rangle$ and $\Sigma_d=\langle \varphi_{0,0,\lambda}\rangle$ with $\lambda^d=1$, $d|(p-1)$.
\end{enumerate}
\end{theorem}
\begin{proof}
Let $G$ be a subgroup of order $pd$ in $\PGU(3,q)$. Two cases are treated separately according as $p>d$ or $p<d$.
Assume first $p>d$. Then Result \ref{res07122023} shows that $G$ has a unique Sylow $p$-subgroup $\Sigma_p$. Moreover, $\Sigma_p$ is a normal subgroup of $G$, and hence $G=\Sigma_p\rtimes \Sigma_d$ where $\Sigma_d$ is a Sylow $d$-subgroup of $G$.  Since any non-trivial element of $\PGU(3,q)$ of order $p$ has exactly one fixed point on $\cH_q(\mathbb{F}_{q^2})$ whereas $\PGU(3,q)$ acts transitively on $\cH_q(\mathbb{F}_{q^2})$, we may assume, up to conjugacy in $\PGU(3,q)$, that $Y_\infty$ is the unique fixed point of $S_p$. As $S_p$ is a normal subgroup of $G$, the point $Y_\infty$ is also fixed by $\Sigma_d$. From $|\cH_q(\mathbb{F}_{q^2})|-1=q^3$, $\Sigma_d$ must have a fixed point $O\in \cH_q(\mathbb{F}_{q^2})$ other than $Y_\infty$. Since $\PGU(3,q)$ is doubly transitive on $\cH_q(\mathbb{F}_{q^2})$ we may assume, up to conjugacy, that $O=(0:0:1)$. Then $\Sigma_d$ is generated by $t=\varphi_{0,0,\lambda}$ with $\lambda^d=1$ where $d|(q^2-1)$. Furthermore, as $\Sigma_p$ is a subgroup of the Sylow subgroup $S_p$ of $\PGU(3,q)$ fixing $Y_\infty$, two cases arise according as $\Sigma_p$ is in the center $Z(S_p)$ of $S_p$ or not. Let $s$ be a generator of $\Sigma_p$.
If $s\in Z(S_p)$ then $s=\varphi_{0,b,1}$ with $b^q-b=0$. Take $\mu\in \mathbb{F}_{q^2}^*$ such that $\mu^{q+1}=b^{-1}$. Then the conjugate of $s$ by $\varphi_{0,0,\mu}$ is $\varphi_{0,0,1}$ while $t$ and $\varphi_{0,0,\mu}$ commute. Therefore, up to conjugacy, $G=\Sigma_p\rtimes \Sigma_d$ with $\Sigma_p=\langle s \rangle$ and $s=\varphi_{0,0,1}$ whereas $\Sigma_d$ and $t$ are as before.
Since $\Sigma_p$ is a normal subgroup of $G$, there exists $i$ with $1\le i \le p-1$ such that $st=ts^i$. A straightforward computation shows that this occurs if and only if $i=1/\lambda^{q+1}$. For $d|(q+1)$, this implies $i=1$, thus $G$ is cyclic and Case (I) occurs. For $d|(q-1)$, we have $i\neq 1$ and hence $G$ is not abelian. From Result \ref{res07122023}, $d|(p-1)$.  Thus Case (II) occurs.
If $s\not\in Z(S_p)$ then $s=\varphi_{a,b,1}$. For $\mu=a^{-1}$, the conjugate of $s$ by $\varphi_{0,0,\mu}$ is $\varphi_{1,b/a^{q+1},1}$ while $t$ and $\varphi_{0,0,\mu}$ commute. Therefore, up to conjugacy, we may  assume $G=\Sigma_p\rtimes \Sigma_d$ where $\Sigma_p=\langle s \rangle$ and $s=\varphi_{1,b/a^{q+1},1}$ while $\Sigma_d$ and $t$ are not changed. Then $st=\phi_{1,b/a^{q+1},\lambda}$ and,  from (\ref{eq08122023}), $ts^i=\phi_{\lambda i,\lambda^{q+1}(\omega (i^2-i)/2+ib/a^{q+1}),\lambda}$. Therefore, $st=ts^i$ if and only if $\lambda i=1$ and $\lambda^{q+1}(\ha \omega(i^2-i)+ib/a^{q+1})=b/a^{q+1}$. The latter condition can also be written as $\ha \omega(i^2-i)=(i^2-i)b/a^{q+1}$, that is, $b=\ha \omega a^{q+1}$ as $\lambda \neq 1$. Therefore, $st=ts^i$ if and only if $s=\varphi_{1,\omega/2,1}$ and $i\lambda=1$. In particular, $G$ is not abelian, and $d|(p-1)$.   This gives Case (III).
Now, assume $p<d$. Then a Sylow $d$-subgroup $\Sigma_d$ of $G$ is a normal subgroup of $G$, and hence $\Sigma_d$ is the unique $d$-subgroup of $G$. As $d$ divides the order of $\PGU(3,q)$, either $d|(q-1)$, or $d|(q+1)$, or $d|(q^2-q+1)$.
Assume that $\Sigma_d$ fixes a point on $\cH_q(\mathbb{F}_{q^2})$. Then $\Sigma_d$ has at least two fixed points, as $|\cH_q(\mathbb{F}_{q^2})|-1$ equals $q^3$. 
Up to conjugacy, we may assume that $\Sigma_d$ fixes $Y_\infty$ and $O$. Then $\Sigma_d=\langle \varphi_{0,0,\lambda}\rangle$ with $\lambda^d=1$. If $\lambda^{q+1}\neq 1$ then $\Sigma_d$ has no any further fixed point, and hence $G$ preserves the pair $\{Y_\infty,O\}$. Since $p>2$ this yields that elements of $G$ of order $p$ fix two points on $\cH_q(\mathbb{F}_{q^2})$ which is not possible. Therefore, $\lambda^{q+1}=1$ and $d|(q+1)$.
This yields that $\Sigma_d$ fixes all points $P=(0,\eta)$ with $\eta^q-\eta=0$, i.e. with $\eta\in \mathbb{F}_q$. Since $\Sigma_d$ is a normal subgroup of $G$, this yields that a generator $s$ of $\Sigma_p$ takes $O$ to a point $P=(0,\eta)$ with $\eta\in\mathbb{F}_q$. But then $s=\varphi_{0,b,1}$ with $b\in \mathbb{F}_q^*$. For $\mu=b^{-1}$, the conjugate of $s$ by $\varphi_{0,0,\nu}$ with $\nu^{q+1}=\mu$ is $\varphi_{0,1,1}$ while a generator $t$ of $\Sigma_d$ and $\varphi_{0,0,\mu}$ commute. Therefore, up to conjugacy, we may  assume $s=\varphi_{0,1,1}$. Also, $st=ts$ and Case (I) occurs.
We are left with the case where $\Sigma_d$ fixes no point on $\cH_q(\mathbb{F}_{q^2})$. Then either $d|(q+1)$, or $d|(q^2-q+1)$. We look at the action of $\PGU(3,q)$ as a projective group of the plane $\PG(2,\mathbb{K})$ where $\mathbb{K}$ is an algebraic closure of $\mathbb{F}_{q^2}$. Then the Hermitian curve $\cH_q$ is left invariant by $\PGU(3,q)$. In particular, $\PGU(3,q)$ preserves both $\cH_q(\mathbb{F}_{q^2})$ and its complementary set in $\PG(2,\mathbb{F}_{q^2})$ whose size equals $q^4+q^2+1-(q^3+1)=q^2(q^2-q+1)$. Furthermore, $\cH_q(\mathbb{F}_{q^2})$ can also be viewed as the set of all isotropic points of a unitary polarity $\pi$ of $\PG(2,\mathbb{F}_{q^2})$.
If $d|(q+1)$ then $\Sigma_d$ fixes a point $R\in \PG(2,\mathbb{F}_{q^2})$ outside $\cH_q(\mathbb{F}_{q^2})$. Let $r$ be the polar line of $R$ w.r.t. $\pi$. Then $r$ is a chord of $\cH_q(\mathbb{F}_{q^2})$. Since $r$ has as many as $q(q-1)$ points other than those on $\cH_q(\mathbb{F}_{q^2})$, there are at least two fixed points on $r$ outside $\cH_q(\mathbb{F}_{q^2})$ under the action of $\Sigma_d$. Since $\Sigma_d$ does not fix $r$ pointwise, these two points, say $R_1, R_2$ are the only fixed points of $\Sigma_d$ on $r$. In particular, $\Sigma_d$ fixes the vertices of the triangle $RR_1R_2$. We show that no more point in  $\PG(2,\mathbb{F}_{q^2})$ is fixed by $\Sigma_d$. In fact, such a further fixed point $T$ of $\Sigma_d$ should lie on a side of the triangle, and that side would be fixed pointwise by $\Sigma_d$. But this is impossible in our case, since the sides of $RR_1R_2$ are chords of $\cH_q(\mathbb{F}_{q^2})$ whereas $\Sigma_d$ is supposed not to fix points on $\cH_q(\mathbb{F}_{q^2})$. Since $\Sigma_d$ is a normal subgroup of $G$, the triangle $RR_1R_2$ is left invariant by $G$. But then $G$ is a contained in a maximal subgroup of $\PGU(3,q)$ whose order equals $6(q+1)^2$. Since $p>3$, this is impossible. A similar geometric approach is used to rule out the other possibility, i.e. $d|(q^2-q+1)$. Look at the action of $\PGU(3,q)$ on $\PG(2,\mathbb{F}_{q^6})$. From $|\cH_q(\mathbb{F}_{q^6})|=q^6+1+q^4(q-1)$ and $|\cH_q(\mathbb{F}_{q^3})|=q^3+1$, the Hermitian curve $\cH_q$ has as many as $q^3(q+1)^2(q-1)$ points in $\PG(2,\mathbb{F}_{q^6})$ but not in $\PG(2,\mathbb{F}_{q^2})$. From $d|(q^2-q+1)$ and $d>3$, $\Sigma_d$ fixes a point $R\in \cH_q(\mathbb{F}_{q^6})$ not lying in $PG(2,\mathbb{F}_{q^2})$. The Frobenius collineation $\mathfrak{f}$ which sends the point $P=(a_1:a_2:a_3)$ to the point $P_{q^2}=\left (a_1^{q^2}:a_2^{q^2}:a_3^{q^2}\right )$ leaves $\cH_q(\mathbb{F}_{q^6})$ invariant. Since $\mathfrak{f}$ and $\Sigma_d$ commute, $\Sigma_d$ also fixes the points $R_{q^2}$ and $R_{q^3}$. Actually, $\Sigma_d$ does not fix another point, otherwise one of the sides, say $\ell$, of the triangle $R R_{q^2}R_{q^4}$ would be fixed by $\Sigma_d$ pointwise. Since $\mathfrak{f}$ takes $\ell$ to another side $r$ of $R R_{q^2}R_{q^4}$ and $\mathfrak{f}$ and $\Sigma_d$ commute, this would yield that $r$ is also fixed pointwise by $\Sigma_d$, which is impossible. As before, this implies that $G$ leaves the triangle invariant $RR_{q^2}R_{q^4}$
invariant. Therefore $G$ is a contained in a maximal subgroup of $\PGU(3,q)$ whose order equals $3(q^2-q+1)$. Since $p>3$, this is impossible.
\end{proof}

\begin{result}\cite[Theorem 5.74]{HKT}
\label{ckt} Let $H$ be a subgroup of $\aut(\mathbb{F}_{q^2}(\cH_q))$ of order $p$. The Galois subcover $\mathbb{F}_{q^2}(\cF'))$ of $\mathbb{F}_{q^2}(\cH_q)$ with respect to $H$ is $\mathbb{F}_{q^2}$-isomorphic to the function field $\mathbb{F}_{q^2}(\xi,\eta)$ where either \rm(I) or \rm(II) hold:
\begin{itemize}
\item[(I)] $\sum_{i=1}^{h}\eta^{\nicefrac{q}{p^i}}+\go \xi^{q+1}=0$ with $\omega^{q-1}=-1$, $\mathfrak{g}(\mathbb{F}_{q^2}(\cF'))=\ha q\left (\frac{q}{p}-1\right )$, and $H$ is in the center of a Sylow $p$-subgroup of $\aut(\mathbb{F}_{q^2}(\cH_q))$;
\item[(II)] $\eta^q+\eta -(\sum_{i=1}^h\ \xi^{\nicefrac{q}{p^i}})^2=0$ for $p>2$, $\mathfrak{g}(\mathbb{F}_{q^2}(\cF'))=\ha \frac{q}{p}(q-1)$, and $H$ is not in the center of a Sylow $p$-subgroup of $\aut(\mathbb{F}_{q^2}(\cH_q))$.
\end{itemize}
\end{result}
The following result is a corollary of \cite[Section 4]{GSX}.
\begin{result}
\label{generi}
Let $\mathfrak{g}$ be the genus of the Galois cover of $\mathbb{F}_{q^2}(\cH_q)$ with respect to a subgroup $G$ of $\aut(\mathbb{F}_{q^2}(\cH_q))$ of order $dp$. Let $S_p$ is a Sylow $p$-subgroup of  $\aut(\mathbb{F}_{q^2}(\cH))$ containing a subgroup $H$ of $G$ of order $p$. Then either
$$\mathfrak{g}=\frac{1}{2} \frac{q}{d}\left (\frac{q}{p}-1\right )\qquad\text{for}\qquad (d,q+1)=1,$$
or 
$$\mathfrak{g}=\frac{1}{2d}(q-d+1)\left (\frac{q}{p}-1\right )\qquad \text{for}\qquad (d,q+1)=d.$$
\end{result}

\section{Galois subcovers of $\mathbb{F}_{q^2}(\cH)$ of type (I) of Result (\ref{dp})}
\label{typeI}
As in Remark \ref{rem250523}, take $\mathbb{F}_{q^2}(\cH_q)$ in its canonical form $\mathbb{F}_{q^2}(x,y)$ with $y^q-y+\omega x^{q+1}=0$ and $\omega^{q-1}=-1$. The group $\Phi=\langle \varphi_{0,1,1}\rangle$ has order $p$, and it is contained in $Z(S_p)$. Let $\eta=y^p-y$ and $\xi=x$. Then $\varphi_{0,1,1}(\eta)=\varphi_{0,1,1}(y^p-y)=\varphi_{0,1,1}(y)^p-\varphi_{0,1,1}(y)=(y+1)^p-(y+1)=y^p-y=\eta$. Moreover, $y^q-y=Tr(\eta)$.
Since $\varphi_{0,1,1}$ fixes $\xi$, this shows that the Galois subcover $\mathbb{F}_{q^2}(\cF')$ of $\mathbb{F}_{q^2}(\cH_q)$ with respect to $\Phi$ is as in (i) of Result \ref{ckt}. That equation can also be written as
\begin{equation}
\label{eqiia}
\sum_{i=0}^{h-1}\ \eta^{p^i}+\go \xi^{q+1}=0.
\end{equation}
Take an element  $r\in \mathbb{F}_{q^2}$ with $r^{d}=1$. Then $\varphi_{0,0,r}$ commutes with $\varphi_{0,1,1}$. Therefore, if $d|(q+1)$ then $\varphi_{0,0,r}$ induces an automorphism $\varphi$ of $\mathbb{F}_{q^2}(\cF')$. More precisely, a straightforward computation shows that $\varphi$ is the map $\varphi:(\xi,\eta)\mapsto (r\xi,\eta)$.  Let $\Phi_{r}$ be the  $\fq$-automorphism group of $\mathbb{F}_{q^2}(\cF')$ generated by $\varphi$. Then
the Galois subcover  of $\mathbb{F}_{q^2}(\cH_q)$ with respect to $G$ of Result (\ref{dp}) of type (I) is the same as the Galois subcover $G_r$ of $\Fs(\cF')$ with respect to $\Phi_r$.

\begin{theorem}
\label{propiia}
The Galois subcover $G_r=\Fs(\zeta,\tau)$ of $\Fs(\cF')$ with respect to $\Phi_r$ has genus $$\mathfrak{g}=\frac{1}{2d}(q-d+1)(\frac{q}{p}-1)$$ and is given by
\begin{equation}
\label{eq1}
\sum_{i=0}^{h-1}\tau^{p^i} +\go \zeta^{\nicefrac{q+1}{d}}=0,\quad d\mid (q+1).
\end{equation}
\end{theorem}
\begin{proof}
We show first that the fixed field $F$ of $\Phi_r$ is generated by $\tau=\eta$ together with
\begin{equation}
\label{eq1iia}
\zeta=\xi^d.
\end{equation} Since $\varphi(\tau)=\tau$ and
$$\varphi(\zeta)=\varphi(\xi^d)=\varphi(\xi)^d=r^{d}\xi^{d}=\xi^{d}=\zeta,$$
we have $\Fs(\zeta,\tau)\subseteq F$. Furthermore, $[\Fs(\cF'):\Fs(\zeta,\tau)]=d$. Since $d$ is prime, this yields either $\Fs(\zeta,\tau)=F$ or $F=\Fs(\cF')$. The latter case cannot actually occur, and hence $F=\Fs(\zeta,\tau)$. Therefore $F=G_r$.
Now, eliminate $\xi$ from Equations (\ref{eqiia}) and (\ref{eq1iia}).
Since $d$ divides $q+1$,  replacing $\xi^{q+1}$  with $(\xi^{\nicefrac{q+1}{d}})^d$ and $\tau=\eta$ in (\ref{eqiia}) gives equation in (\ref{eq1}). The formula for the genus follows from \cite[Lemma 12.1(iii)(b)]{HKT}.
\end{proof}
\section{Galois subcovers of $\mathbb{F}_{q^2}(\cH)$ of type (II) of Result (\ref{dp})}
We keep our notation up from Section \ref{typeI}. 
Assume that $d$ divides $p-1$, and take $r\in \mathbb{F}_p^*$ with $r^{d}=1$. Then $\varphi_{0,0,r}^{-1}\circ \varphi_{0,1,1}\circ \varphi_{0,0,r}=\varphi_{0,r^2,1}\in \langle \varphi_{0,1,1}\rangle$, and hence $\varphi_{0,0,r}$ induces an automorphism $\varphi$ of $\mathbb{F}_{q^2}(\cF')$. Here,  $\varphi$ is the map $\varphi:(\xi,\eta)\mapsto (r\xi,r^2\eta)$.  Let $\Phi_{r}$ be the  $\fq$-automorphism group of $\mathbb{F}_{q^2}(\cF')$ generated by $\varphi$. Then
the Galois subcover  of $\mathbb{F}_{q^2}(\cH_q)$ with respect to $G$ of Result (\ref{dp}) of type (II) is the Galois subcover $G_r$ of $\Fs(\cF')$ with respect to $\Phi_r$.
\begin{theorem} The Galois subcover  $G_r=\Fs(\epsilon,\rho)$ of $\Fs(\cF')$ with respect to $\Phi_r$ has equation 
\begin{equation}
\label{eq2}
\go\epsilon^{\nicefrac{(q-1)}{d}}-A(\epsilon,\rho)=0, \quad d\mid (p-1)
\end{equation}
where
$$A(\epsilon,\rho)=\rho+\epsilon^{\nicefrac{2(p-1)}{d}}\rho^p+\cdots+\epsilon^{\nicefrac{2(p^{h-1}-1)}{d}}\rho^{\nicefrac{q}{p}}.$$
\end{theorem}

\begin{proof}
We show first that the fixed field $F$ of $\Phi_r$ is generated by
    \begin{equation}
    \label{inv1}
    \epsilon=\xi^d
    \end{equation}
    together with
\begin{equation}
\label{inv2}
\rho=\frac{\eta}{\xi^2}.
\end{equation}
Since $$\varphi(\epsilon)=\varphi(\xi^d)=\varphi(\xi)^d=r^d\xi^d=\xi^d=\epsilon$$ and
$$\varphi(\rho)=\frac{\varphi(\eta)}{\varphi(\xi^2)}=\frac{\varphi(\eta)}{\varphi(\xi)^2}=\frac{r^2\eta}{r^2\xi^2}=\frac{\eta}{\xi^2}=\rho,$$
we have $\Fs(\epsilon,\rho)\subseteq F$. Furthermore, $[\Fs(\cF'):\Fs(\epsilon,\rho)]=d$. Since $d$ is prime, this yields either $\Fs(\epsilon,\rho)=F$ or $F=\Fs(\cF')$. The latter case cannot actually occur, and hence $F=\Fs(\epsilon,\rho)$. Therefore $F=G_r$.
We have to eliminate $\xi$ and $\eta$ from equations (\ref{eqiia}), (\ref{inv1}) and (\ref{inv2}).
From (\ref{inv2}) we have $\eta=\rho\xi^2$ then $Tr(\eta)=Tr(\rho\xi^2)$. This yields that
\begin{equation}
    \label{pass1}
    Tr(\eta)=\xi^2\rho+\xi^{2p}\rho^p+\cdots+\xi^{\nicefrac{2q}{p}}\rho^{\nicefrac{q}{p}},
\end{equation}
    whence
    \begin{equation}
        \label{pass2}
    Tr(\eta)=\xi^2\big(\rho+\xi^{2(p-1)}\rho^p+\cdots+\xi^{2(\nicefrac{q}{p}-1)} \rho^{\nicefrac{q}{p}}\big).
    \end{equation}
    Since $d$ divides $p-1$, $Tr(\eta)$ in (\ref{pass1}) can also be written as
   \begin{equation}
        \label{pass3}
        Tr(\eta)=\xi^2\big(\rho+(\xi^d)^{\nicefrac{2(p-1)}{d}}\rho^p+\cdots+(\xi^d)^{\nicefrac{2(p^{h-1}-1)}{d}}\rho^{\nicefrac{q}{p}}\big).
    \end{equation}
    Therefore
    \begin{equation}
        \label{pass4}
        Tr(\eta)=\xi^2\big(\rho+\epsilon^{\nicefrac{2(p-1)}{d}}\rho^p+\cdots+\epsilon^{\nicefrac{2(p^{h-1}-1)}{d}}\rho^{\nicefrac{q}{p}}\big)=\xi^2 A(\epsilon,\rho).
    \end{equation}
    This, together with (\ref{eqiia}), give
    \begin{equation}
        \label{pass5}
        \go\xi^{q+1}=\xi^2 A(\epsilon,\eta).
    \end{equation}
    Since $d\mid (p-1)$ the number $\frac{q-1}{p-1}$ is an integer. Thus Equation
 (\ref{eq2}) follows from (\ref{pass5}).
\end{proof}

\section{Galois subcovers of $\mathbb{F}_{q^2}(\cH)$ of type (III) of Result (\ref{dp})} 
This time, take $\mathbb{F}_{q^2}(\cH_q)$ in its canonical form $\mathbb{F}_{q^2}(x,y)$ with $y^q+y-x^{q+1}=0$.
The group $\Psi=\langle \psi_{1,\nicefrac{1}{2},1}\rangle$ has order $p$, and it is not contained in $Z(S_p)$. Let $\xi=x^p-x$ and $\eta=y-\ha x^2$. A straightforward computation shows that $\psi_{1,\nicefrac{1}{2},1}(\xi)=\xi$ and $\psi_{1,\nicefrac{1}{2},1}(\eta)=\eta$. Moreover,
$$y^q+y-x^{q+1}=\eta^q+\ha x^{2q}+\eta+\ha x^2-x^{q+1}=\eta^q+\eta+\ha x^{2q}+\ha x^2-x^{q+1}=\eta^q+\eta+\ha (x^q-x)^2.$$
Since $Tr(\xi)=x^q-x$, this gives
$$\eta^q+\eta+\ha(x^q-x)^2=\eta^q+\eta+\ha Tr(\xi)^2.$$
Therefore, the Galois subcover $\mathbb{F}_{q^2}(\cF')$ of $\mathbb{F}_{q^2}(\cH_q)$ with respect to $\Psi$ is
$\mathbb{F}_{q^2}(\xi,\eta)$ with
\begin{equation}
\label{eqiib}
\eta^q+\eta+\ha \left(\sum_{i=1}^{h}\xi^{p^{i-1}}\right)^2=0.
\end{equation}
In particular, $\mathbb{F}_{q^2}(\cF')$ is $\mathbb{F}_{q^2}$-isomorphic to (II) of Result \ref{ckt}.
Assume that $d$ divides $p-1$, and take $r\in \mathbb{F}_p^*$ with $r^{d}=1$. Then $\psi_{0,0,r}^{-1}\circ \psi_{1,\nicefrac{1}{2},1}\circ \psi_{0,0,r}=\varphi_{r,\nicefrac{1}{2}r^2,1}\in \langle \varphi_{1,\nicefrac{1}{2},1}\rangle$, and hence $\varphi_{0,0,r}$ induces an automorphism $\varphi$ of $\mathbb{F}_{q^2}(\cF')$. Moreover, $\psi$ is the map $\psi:(\xi,\eta)\mapsto (r\xi,r^2\eta)$.  Let $\Psi_{r}$ be the  $\fq$-automorphism group of $\mathbb{F}_{q^2}(\cF')$ generated by $\psi$. Then
the Galois subcover  of $\mathbb{F}_{q^2}(\cH_q)$ with respect to $G$ of Result (\ref{dp}) of type (III) is the Galois subcover $G_r$ of $\Fs(\cF')$ with respect to $\Psi_r$.
\begin{theorem}
\label{propiiaA}
The Galois subcover $G_r=\Fs(\iota,\nu)$ of $\Fs(\cF')$ with respect to $\Psi_r$ has genus $$\mathfrak{g}=\frac{q}{2dp}(q-1)$$ and is given by 
\begin{equation}
\label{eq3}
\left (\frac{\tau^2}{\iota^d}\right )^{\nicefrac{(q-1)}{d}}+1-A(\iota,\tau)=0
\end{equation}
where 
\begin{equation*}
    A(\iota,\tau)=\sum_{i=0}^{h-1}\sum_{j=0}^{h-1}\left (\frac{\tau^2}{\iota^d}\right )^{\nicefrac{(p^i-1)}{2d}}\left (\frac{\tau^2}{\iota^d}\right)^{\nicefrac{(p^j-1)}{2d}}\iota^{\nicefrac{(p^i+p^j)}{2}}.
\end{equation*}
\end{theorem}

\begin{proof}
We show first that the fixed field $F$ of $\Psi_r$ is generated by
    \begin{equation}
    \label{inv3}
    \nu=\eta^d,
    \end{equation}
    together with
\begin{equation}
\label{inv4}
\iota=\frac{\xi^2}{\eta},
\end{equation}
and 
\begin{equation}
\label{inv5}
\tau={\xi^d}.
\end{equation}
Since
$$\varphi(\nu)=\varphi(\eta)^d=r^{2d}\eta^d=\eta^d=\nu,\quad \varphi(\tau)=\varphi(\xi)^d=r^{d}\xi^d=\xi^d=\tau  $$
and
$$\varphi(\iota)=\varphi\left(\frac{\xi^2}{\eta}\right)=\frac{\varphi(\xi^2)}{\varphi(\eta)}=\frac{\varphi(\xi)^2}{\varphi(\eta)}=\frac{(r\xi)^2}{r^2\eta}=\frac{r^2\xi^2}{r^2\eta}=\frac{\xi^2}{\eta},$$
we have $\Fs(\iota,\nu,\tau)\subseteq F$. Furthermore, 
$[\mathbb{F}_{q^2}(\iota,\nu,\tau)(\xi):\mathbb{F}_{q^2}(\iota,\nu,\tau)]=d$ 
and $\eta\in \mathbb{F}_{q^2}(\iota,\nu,\tau)(\xi)$. Therefore, $[\Fs(\cF'):\Fs(\iota,\nu,\tau)]\le d$. Since $d$ is prime, this yields either $\Fs(\iota,\mu,\nu)=F$ or $F=\Fs(\cF')$. The latter case cannot actually occur, and hence $F=\Fs(\iota,\mu,\nu)$.
Therefore $F=G_r$.
We go on by eliminating $\xi$ and $\eta$ from Equations (\ref{eqiib}), (\ref{inv3}), (\ref{inv4}) and (\ref{inv5}). From the definition of the trace of $\xi$, 
$Tr(\xi)^2=(\xi+\cdots+\xi^{\nicefrac{q}{p}})^2.$
By a straightforward computation, 
$$Tr(\xi)^2=\sum_{i=0}^{h-1}\sum_{j=0}^{h-1} \xi^{p^i+p^j}.$$
This can also be written as 
\begin{equation}
\label{pass6}
Tr(\xi)^2=\sum_{i=0}^{h-1}\sum_{j=0}^{h-1} (\xi^2)^{\nicefrac{(p^i+p^j)}{2}}.
\end{equation}
From (\ref{inv4}), $\xi^2=\eta\iota$. Therefore, in (\ref{pass6}) the square trace of $\xi$ is equal to
\begin{equation}
    \label{pass7}
    \sum_{i=0}^{h-1}\sum_{j=0}^{h-1} (\eta\iota)^{\nicefrac{(p^i+p^j)}{2}}.
\end{equation}
Since $\ha(p^i+p^j)-1=\ha(p^i-1)+\ha(p^j-1)$, the sum in $(\ref{pass7})$ turns out to be equal to 
\begin{equation}
    \label{pass7bis}
    \eta\sum_{i=0}^{h-1}\sum_{j=0}^{h-1}\eta^{\nicefrac{(p^i-1)}{2}}\eta^{\nicefrac{(p^j-1)}{2}}\iota^{\nicefrac{(p^i+p^j)}{2}}.
\end{equation}
As $d$ divides $\ha(p^i+p^j-2)$ and $2$ divides both $p^i-1$ and $p^j-1$, the sum in (\ref{pass7}) equals 
\begin{equation}
    \label{pass8}
    \eta\sum_{i=0}^{h-1}\sum_{j=0}^{h-1}(\eta^d)^{\nicefrac{(p^i-1)}{2d}}(\eta^d)^{\nicefrac{(p^j-1)}{2d}}\iota^{\nicefrac{(p^i+p^j)}{2}}
\end{equation}
by replacing $\eta^d$ with $\nu$
\begin{equation}
    \label{pass9}
     \eta\sum_{i=0}^{h-1}\sum_{j=0}^{h-1} \nu^{\nicefrac{(p^i-1)}{2d}}\nu^{\nicefrac{(p^j-1)}{2d}}\iota^{\nicefrac{(p^i+p^j)}{2}}
\end{equation}
Let
\begin{equation}
    \label{pass10}
     A(\iota,\nu)=\sum_{i=0}^{h-1}\sum_{j=0}^{h-1} \nu^{\nicefrac{(p^i+p^j-2)}{2d}}\iota^{\nicefrac{(p^i+p^j)}{2}}.
\end{equation}
Therefore, $\eta^q+\eta=\eta A(\iota,\nu)$, and dividing both sides by $\eta$ gives $\eta^{q-1}+1= A(\iota,\nu)$. Since $d$ divides $q-1$, replacing $\eta^d$ by $\nu$ shows 
\begin{equation}\label{nuiota}
\nu^{\nicefrac{(q-1)}{d}}+1-A(\iota,\nu)=0.
\end{equation}
From (\ref{inv3}), (\ref{inv4}), and (\ref{inv5}), 
$\nu=\frac{\tau^2}{\iota^d}$. Now the claim follows from   (\ref{nuiota}).
\end{proof}


\section{Weierstrass semigroups and application to AG-codes}
We compute the Weierstrass semigroup at the unique place centred at the point at infinity of some of the maximal curves considered in the present paper.
\begin{prop}
\label{WS}
Let $P_\infty$ be the unique point at infinity of the following two curves
\begin{equation}
        \label{INTERM1}
        \sum_{i=1}^{h} Y^{\nicefrac{q}{p^i}}+\go X^{q+1}=0,\quad \omega^{q-1}=-1, h\geq 2;\qquad
\sum_{i=1}^{h}Y^{\nicefrac{q}{p^i}} +\go X^{\nicefrac{(q+1)}{d}}=0, \omega^{q-1}=-1,d\mid (q+1).
\end{equation} Then the Weierstrass semigroup at $P_\infty$ is generated by $\textstyle\frac{q}{p}$, and $q+1$, respectively by  $\textstyle\frac{q}{p}$, and $\frac{q+1}{d}$.
\end{prop}
\begin{proof}
 Fore the first equation the claim follows from the remark after the proof of Lemma 12.2 in \cite{HKT} applied for $n=h-1$ and $m=q+1$.\\
For the second equation the claim follows from the remark after the proof of Lemma 12.2 in \cite{HKT} applied for $n=h-1$ and $m=\frac{q+1}{d}$.
 \end{proof}
Let $S$ be a numerical semigroup. The gaps of $S$ are the elements in $\mathbb{N}\setminus S$. The number $g(S)$ of gaps of $S$ is the {\em{genus}} of $S$. If $S$ is the Weierstrass semigroup of a curve at a point then $g(S)$ coincides with the genus of the curve.  Let $(a_1,\ldots,a_k)$ be a sequence of positive integers such that their greatest common divisor is 1. Let $d_0=0,\,d_i=g.c.d.(a_1,\ldots,a_i)$ and $A_i=\left\{\frac{a_1}{d_i},\ldots,\frac{a_i}{d_i}\right \}$ for $i=1,\ldots,k$. Let $S_i$ be the semigroup generated by $A_i$. The sequence $(a_1,\ldots,a_k)$ is {\em{telescopic}} whenever $\frac{a_i}{d_i}\in S_{i-1}$ for $i=2,\ldots,k$. A  {\em{telescopic semigroup}} is a numerical semigroup generated by a telescopic sequence.
\begin{result}\label{resgeneresemigroup}\cite[Lemma 6.5]{KirfelPellikaan1995}
 For the semigroup generated by a telescopic sequence $(a_1,\ldots,a_k)$, let
 $$l_g(S_k):=\sum_{i=1}^{k} \left(\frac{d_{i-1}}{d_i}-1 \right )a_1,\quad
 g(S_k):=\frac{l_g(S_k)+1}{2}.$$
\end{result}
\begin{theorem}
\label{WS1} Let $P_\infty$ be the unique point of infinity of the curves $\bar{\cH}_q$  in Theorem \ref{th1}. Then the Weierstrass semigroup $H(P_\infty)$ has the following properties: 
\begin{itemize}
\item[\rm(I)] $H(P_\infty)=\langle \frac{q}{p},q+1\rangle$, for the curve of Equation (I);
\item[\rm(II)] $\frac{q}{p},\frac{q-1}{d}\in H(P_\infty)$,  for the curve of Equation (II);
\item[\rm(III)] $\frac{2(q-1)}{d},q-1\in H(P_\infty)$, for the curve of Equation (III). 
\end{itemize}
\end{theorem}
\begin{proof}
Case (i). 
The pole numbers of $x$ and $y$ at $P_\infty$ are $q$ and $2q/p$, respectively. Since the curve is $\mathbb{F}_{q^2}$-maximal and $P_\infty$ is an $\mathbb{F}_{q^2}$-rational point, $q+1\in H(P_\infty)$; see \cite[Theorem 10.6]{HKT}. Let $d_0=0$, $d_1=2\frac{q}{p}$, $d_2=\frac{2}{q}$ and $d_3=1$, and $A_1=\{1\}$, $A_2=\{2,p\}$, $A_3=\{2\frac{q}{p},q,q+1\}$. Then $p\in S_1$ and $q+1\in S_2$. 
Thus the sequence $\{2\frac{q}{p},q,q+1\}$ is telescopic. Furthermore,
$$l_g(S_3)=-\frac{2q}{p}+q+(\frac{q}{p}-1)(q+1)=\frac{q^2}{p}-\frac{q}{p}-1,$$ whence the claim follows by Result (\ref{resgeneresemigroup}).

Case (ii). 
From Equation (II), $\left[\mathbb{F}_{q^2}(\bar{\cH}_q):\mathbb{F}_{q^2}(x)\right]=\frac{q}{p}$ and 
$\left[\mathbb{F}_{q^2}(\bar{\cH}_q):\mathbb{F}_{q^2}(y)\right]=\frac{q-1}{d}.$
Therefore, $\frac{q}{p}$ and $\frac{q-1}{d}$ are non-gaps at $P_{\infty}$.

Case (iii). The above argument applied to the curve $\bar{\cH}_q$ of Equation (III),  shows that $\frac{q-1}{d}$ and $\frac{q}{p}$ are non-gaps of $G_r$ at $P_{\infty}$.
\end{proof}
Let $C$ denote any $\mathbb{F}_{q^2}$-maximal curve equipped with an $\mathbb{F}_{q^2}$-rational point $P$. Let $D$ be a set of $\mathbb{F}_{q^2}$-rational points of $C$ other than $P$. From previous work by Janwa \cite{jan} and Garcìa-Kim-Lax \cite{GL}, if the divisor $G$ is taken as multiple of $P$ then knowledge of the gaps at $P_\infty$ may allow one to show that the minimum distance of the resulting evaluation code $C_L(G,D)$ or differential code $C_\Omega(G,D)$ may be better than the designed minimum distance of that code. In particular, it is shown in \cite{GKL} that $t$ consecutive gaps at $P$ (under some conditions on the order sequence at $P$) gives a minimum distance $d$ of the code at least $t$  greater than the designed minimum distance. This motivates to investigate large intervals of gaps at the point $P_\infty$ of the $\mathbb{F}_{q^2}$-maximal curves considered in the present paper. Here we limit ourselves to show a couple of experimental results. We use Janwa's result as stated in \cite[Theorem 2]{GKL} together with \cite[Theorem 3]{GKL} for the zero divisor $B=0$.  
\begin{example}
Take the curve of equation (\ref{eqthI}) for $\mathcal{C}$, and let $p=7,d=5,h=2$. Then $d\mid (q+1)=7^2+1$. Form Proposition \ref{WS}, the non-gaps at $P_{\infty}$ are $q/p=7$ and $(q+1)/d=10$. 
The gap sequence at $P_\infty$ is
$1,2,3,4,5,6,8,9,11,12,13,15,16,18,19,22,23,25,26,29,32,33,36,39,43,46,53.$
Each of the integers $11=\gamma-2=\gamma-t$, $12=\gamma -1$ and $13=\gamma$ is a gap at $P_{\infty}$. From \cite[Theorem 3]{GKL}, the minimum distance of the code $C_L(\gamma P_\infty,D)$ is at least $d^*=|D|-\gamma+t+1=5037$ whereas the designed minimum distance is $d'=|D|-\gamma=5034$. 
\end{example}
\begin{example}
Take the curve of equation (\ref{eqthI}) for $\mathcal{C}$, and let $p=5,d=3,h=3$. Then $d\mid (q+1)=5^3+1$. Form Proposition \ref{WS}, the non-gaps at $P_{\infty}$ are $q/p=25$ and $(q+1)/d=42$. 
The gap sequence at $P_\infty$ is
$1,\ldots,24,26,\ldots,41,43,\ldots,66,68,\ldots\ldots,920,922,\ldots,962,964,\ldots,981,983.$
Each of the integers $1022=\alpha,\ldots,1030=\alpha+8=\alpha+t$ and $1072=\beta,\ldots,1063=\beta-t=\beta-(t-1)$ is a gap at $P_{\infty}$. From \cite[Theorem 4]{GKL}, the minimum distance of the differential code $C_\Omega(\gamma P_\infty,D)$ with $\gamma=\alpha+\beta-1$ is at least $d^*=\alpha+\beta-1-(2\mathfrak{g}-2)+(t+1)=1120$ whereas  the designed minimum distance is $d'=\alpha+\beta-1-(2\mathfrak{g}-2)=1112$.
\end{example}
It may be noticed that the curve of equation (\ref{eqthI}) is a $C_{ab}$-curve with $a=q/p$ and $b=(q+1)/d$. Evaluation codes defined over a $C_{ab}$-curve have been the subject of  recent papers where both encoding and decoding problems are also treated; see \cite{BRS}.

\end{document}